\documentclass{amsart}
\usepackage{amsfonts,amssymb,amsbsy,amsmath,amsthm}
\usepackage{epsfig,float,bm}
\usepackage{etex,pictex,graphicx,hyperref,relsize,esvect,cancel,ifthen}
\usepackage[usenames]{xcolor}
\usepackage{hyperref}
\usepackage{fontawesome5}
\usepackage{caption}
\newtheorem{theorem}{Theorem}

\newtheorem{corollary}[theorem]{Corollary}
\newtheorem{proposition}[theorem]{Proposition}

\theoremstyle{definition}

\newtheorem{remark}[theorem]{Remark}
\newcommand{\bfi}{\mathbf{i}}
\newcommand{\bfj}{\mathbf{j}}
\DeclareMathOperator{\acl}{acl}

\newcommand{\cnl}{\operatorname{cl}}
\newcommand{\snl}{\operatorname{sl}}
\newcommand{\slh}{\operatorname{slh}}
\hypersetup{
    colorlinks=true,
    citecolor=blue,
    linkcolor=blue,
    filecolor=magenta,      
    urlcolor=blue,
    }

\begin{document}

\title[A Remarkable Relation Between the Squircle and Lemniscate]{An Elementary Proof of a Remarkable Relation Between the Squircle and Lemniscate}
\author{Zbigniew Fiedorowicz, Muthu Veerappan Ramalingam}
\address{\hbox{\parbox{150pt}{Department of Mathematics,\\
The Ohio State University\\
Columbus, Ohio 43235 -1174, USA}\hspace{80pt}
\parbox{\linewidth}{Aranthangi, Tamil Nadu, India}}
}

\begin{abstract}
It is well known that there is a somewhat mysterious relation between the area of the quartic Fermat curve $x^4+y^4=1$, aka squircle, and the arc length of the lemniscate
$(x^2+y^2)^2=x^2-y^2$. The standard
proof of this fact uses relations between elliptic integrals and the gamma function. In this article we generalize this result to relate areas of sectors
of the squircle to lengths of arcs of the lemniscate. We provide a geometric interpretation of this relation and an elementary proof
of the relation, which only uses basic integral calculus. We also discuss an analytic interpretation relating analogs of trigonometric
functions associated with the lemniscate and the squircle, as well as
an alternate version of this geometric relation, which is implicit in a calculation of Siegel.
\end{abstract}

\maketitle

\section{Introduction}

This article proves a geometric result which relates two different analogs of the trigonometric functions. Recall that there are two common
rigorous constructions of the trigonometric functions using integrals as the starting point. One is based on arc
length while the other is based on area.

The first construction relates to uniform circular motion. A particle moving at unit speed counterclockwise
around the unit circle $x^2+y^2=1$ starting at the point $(1,0)$ will be at position $(x,y)=(\cos(t),\sin(t))$ at time $t$.
It follows that the arc length of the circular sector from $(1,0)$ to a pont $(x,y)$ in the first quadrant will equal the
length of time the particle takes to travel that arc and hence we have
\begin{equation}
t=\arccos(x)=\int_x^1\frac{1}{\sqrt{1-u^2}}\,du.
\end{equation}
Note that this amounts to the definition of radian measure of the subtended angle.
We can use this integral formula to rigorously define the function ${\arccos:[0,1]\longrightarrow\left[0,\frac{\pi}{2}\right]}$. 
We can then define the $\cos$ function on the interval $\left[0,\frac{\pi}{2}\right]$ as the inverse of this function. Then
we can extend the domain of cosine to all real numbers by using symmetries of the circle and periodicity.

An alternative approach to a rigorous construction of the trigonometric functions is used in the calculus textbooks
of Spivak \cite[Chap. 15]{spivak} and Apostol \cite[p. 102]{apostol}. This approach is based on the observation that the radial area swept out by the above-mentioned
particle moving along the unit circle is half of the arc length swept out by the particle. Thus we obtain
\begin{equation}
\frac{1}{2}t=\frac{1}{2}\arccos(x)=\int_x^1\sqrt{1-u^2}\,du+\frac{1}{2}x\sqrt{1-x^2}.
\end{equation}
We then use this formula to define $\arccos:[0,1]\longrightarrow\left[0,\frac{\pi}{2}\right]$.  Taking inverses and
using symmetry and periodicity, we obtain a rigorous construction of the cosine function. An advantage of this approach
is that it avoids the use of improper integrals.

These two constructions serve as templates for defining analogs of trigonometric functions using other planar curves
instead of the circle. For certain special curves this defines functions with interesting mathematical properties.

In the late 18th century Euler and Gauss observed that the first approach to the trigonometric functions had a remarkable
analog obtained by replacing the unit circle by the lemniscate $(x^2+y^2)^2=x^2-y^2$ (\cite{Wikipedia:lemniscate_of_bernoulli}), 
whose graph is a twisted circle.
Considering a particle moving at unit speed around the lemniscate, they defined an analog of $\arccos$ by the following
integral formula for arc length:
\begin{equation}
t=\acl(x)=\int_x^1\frac{1}{\sqrt{1-u^4}}\,du.
\end{equation}
Taking the inverse of this function and using symmetries of the lemniscate and periodicity, they defined lemniscate
analogs of the trigonometric functions. They then observed that these functions satisfy various nice identities analogous to 
trigonometric identities. Further generalizations by Legendre, Jacobi, Abel and Weierstrass led to the theory of elliptic integrals,
 functions and curves, which have important applications to algebraic geometry, number theory, mathematical physics and cryptography.

A different analog of the trigonometric functions, based upon the second approach, stems from the observaton that the
notion of vector length in the plane can be generalized to the following formula
\[||x\bfi+y\bfj||_p=\sqrt[p]{|x|^p+|y|^p},\]
where $p$ is a real number $\ge 1$.  These alternative notions of length are known as $p$-norms. The case $p=2$ is the usual Euclidean notion of length in the plane. These different notions of
length give different notions of geometry in the plane. In particular in the $p$-norm geometry, the unit circle has the
equation $|x|^p+|y|^p=1$, also known as a Lam\'e curve.  We can then use the second approach to define corresponding generalizations of the trigonometric
functions. Consider a particle moving counterclockwise around $|x|^p+|y|^p=1$ starting at $(1,0)$ in such a way that the radial area 
is swept out at a constant rate of $\frac{1}{2}$. [Such a motion is called \textit{keplerian} after Kepler's second law of planetary
motion, c.f. \cite{gambini_et_al}.]  We define the functions $\cos_p$ and $\sin_p$ as giving the position 
$(x,y)=(\cos_p(t),\sin_p(t))$ at time $t$.  Then if $(x,y)$ is in the first quadrant, we have
\[\frac{1}{2}t=\frac{1}{2}\arccos_p(x)=\int_x^1\sqrt[p]{1-u^p}\,du+\frac{1}{2}x\sqrt[p]{1-x^p}.\]
This approach has been studied by various authors, e.g \cite{gambini_et_al}, \cite{grammel}, \cite{alevin}, \cite{squigonometry_book},
\cite{poodiack}, \cite{shelupsky}, \cite{wood}.
In the special case $p=4$ the graph of $x^4+y^4=1$ is a flatened circle,
popularly known as a \textit{squircle} (\cite{Wikipedia:squircle}), and we have
\begin{equation}
\frac{1}{2}t=\int_x^1\sqrt[4]{1-u^4}\,du+\frac{1}{2}x\sqrt[4]{1-x^4}.
\end{equation}

Calculations of Legendre and Dirichlet in the early 19th century established the following relation between special cases of the
integrals (3) and (4):
\begin{equation}\label{lemnisquircle_eqn}
\int_0^1\frac{1}{\sqrt{1-u^4}}\,du = \sqrt{2}\int_0^1\sqrt[4]{1-u^4}\,du,
\end{equation}
which is an analog of the following trigonometric relation
\[\int_0^1\frac{1}{\sqrt{1-u^2}}\,du = 2\int_0^1\sqrt{1-u^2}\,du.\]
If we multiply these relations by 4 we obtain the geometric interpretations that the perimeter of the lemniscate is $\sqrt{2}$
times the area enclosed by the squircle, whose trigonometric analog is that the circumference of the unit circle is twice the area
enclosed by the circle. This suggests that there might be a more general relation between lengths of arcs of the lemniscate
and areas of radial sectors of the squircle and consequently between the corresponding trigonometric analogs. In this article we demonstrate such a relation.

This relation is vividly illustrated in the following YouTube \href{https://www.youtube.com/watch?v=mAzIE5OkqWE&t=3s}{video} created by the second author (c.f. \cite{mvr}), corresponding to Figure 1 below.\\ 
\begin{figure}[H]
\centerline{
\includegraphics[scale=0.1]{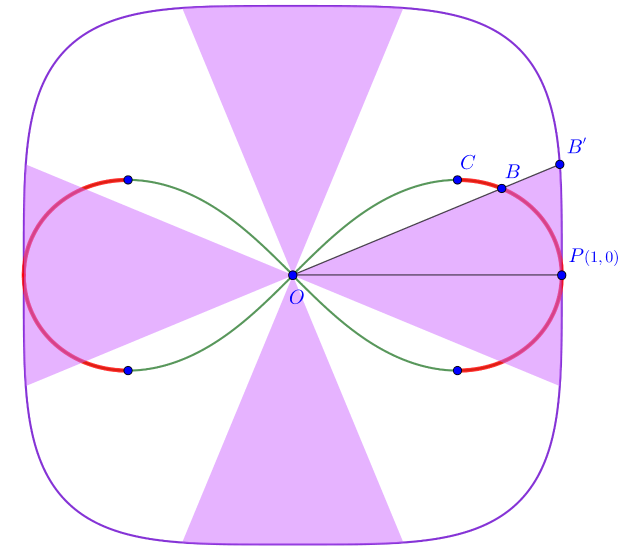}
}
\caption{Squircle and Lemniscate}
\end{figure}
The relation depicted here
is $L=A\sqrt{2}$, where $A$ is the shaded area in the picture and $L$ is the arc length of the bolded portion of the lemniscate. The points $B$ and $C$ in the picture are related by the equation $\overline{OC}=\overline{OB}^2$, where $O$ is the origin.
Due to the symmetries of the squircle and the lemniscate it suffices to check the relation for the portions of the areas and lengths lying in the first
octant: ${\{(x,y)\ |\ 0\le y\le x\}}$. If we denote the corresponding area by $a$ and the corresponding arc length by $l$, we have $A=8a$ and
$L=4l$. Thus the given relation reduces to
\[4l=8a\sqrt{2}\quad\Longleftrightarrow\quad l=2a\sqrt{2}\quad\Longleftrightarrow\quad l-2a\sqrt{2}=0.\]
Hence we are reduced to proving Theorem \ref{main_result} below.

\bigskip

\textit{The authors would like to take this opportunity to thank the editor and referees for their thoughtful review of our submission
and for their helpful suggestions.}

\section{Main Result}

\begin{theorem}\label{main_result} Let $B$ be a point in the first quadrant of the lemniscate $\left(x^2+y^2\right)^2=x^2-y^2$ and let $B'$ 
be its radial projection onto the squircle $x^4+y^4=1$. Let $C$ be the point on the first quadrant of the lemniscate such that 
$\overline{OC}=\overline{OB}^2$. Then
\[l-2a\sqrt{2}=0,\]
where $l$ is the arc length of the lemniscate from $C$ to $P$ and $a$ denotes the area of the squircular sector $OPB'$.
\end{theorem}

\begin{proof}
We will use polar coordinates to compute these areas and arc lengths. The relevant formulas are
\begin{eqnarray*}
\mbox{area} &= &\frac{1}{2}\int_{\theta_1}^{\theta_2}r^2\,d\theta\\
\mbox{arc length} &= &\int_{\theta_1}^{\theta_2}\sqrt{r^2+\left(\frac{dr}{d\theta}\right)^2}\,d\theta
\end{eqnarray*}
The integrals for the arc length of the lemniscate and the area of the squircle do not have elementary
antiderivatives. Hence our strategy will be to express $l-2a\sqrt{2}$ in terms of these integrals, differentiate this expression,
and then, by applying the fundamental theorem of calculus, show that this derivative is 0. This implies $l-2a\sqrt{2}$
is a constant, and evaluating this constant at the initial position $B=C=P$ we obtain that $l-2a\sqrt{2}=0$.

The polar equation of the lemniscate in the first quadrant is
\[r=\sqrt{\cos(2\theta)},\ 0\le\theta\le\frac{\pi}{4}.\]
To derive the polar equation of the  squircle, we start with the obvious Cartesian parametric equations over the interval
$\left[0,\frac{\pi}{2}\right]$:
\begin{eqnarray*}
x &= &\sqrt{\cos(s)}\\
y &= &\sqrt{\sin(s)}
\end{eqnarray*}
and convert to polar coordinates using the relations:
\[\tan(\theta)=\frac{y}{x},\qquad r^2=x^2+y^2.\]
We obtain
\[\tan(\theta)=\frac{\sqrt{\sin(s)}}{\sqrt{\cos(s)}}=\sqrt{\tan(s)}\quad\Longrightarrow\quad\tan(s)=\tan^2(\theta).\]
To relate $\cos(s)$ and $\sin(s)$ to $\theta$, we use right triangle relations.\\
\centerline{
\beginpicture
\setcoordinatesystem units <2cm,2cm>
\setplotarea x from 0 to 3, y from -0.2 to 1
\setlinear
\plot 0 0 2 0 /
\plot 2 0 2 1 /
\plot 0 0 2 1 /
\put{$s$} at 0.4 0.1
\put{1} at 1 -0.1
\put{$\tan^2(\theta)$} at 2.35 0.4
\put{$\sqrt{1+\tan^4(\theta)}$} at 0.4 0.6
\endpicture
}\\
We obtain
\begin{eqnarray*}
x &= &\sqrt{\cos(s)}=\frac{1}{\sqrt[4]{1+\tan^4(\theta)}}\\
y &= &\sqrt{\sin(s)}=\frac{\tan(\theta)}{\sqrt[4]{1+\tan^4(\theta)}}.
\end{eqnarray*}
Hence a polar equation of the  squircle is
\begin{equation}\label{polar_tan}
r^2=x^2+y^2=\frac{1+\tan^2(\theta)}{\sqrt{1+\tan^4(\theta)}}=\frac{\sec^2(\theta)}{\sqrt{1+\tan^4(\theta)}}.
\end{equation}
In order to compare the squircle with the lemniscate, whose polar equation involves $\cos(2\theta)$, it is sensible
to rewrite this equation in terms of $\cos(\theta)$ and $\sin(\theta)$ and then simplify it using the
half-angle relations:
\begin{eqnarray}
r^2 &= &\frac{\frac{1}{\cos^2(\theta)}}{\sqrt{1+\frac{\sin^4(\theta)}{\cos^4(\theta)}}}\\
&= &\frac{1}{\sqrt{\cos^4(\theta)+\sin^4(\theta)}}\\
&= &\frac{1}{\sqrt{\left(\frac{1+\cos(2\theta)}{2}\right)^2+ \left(\frac{1-\cos(2\theta)}{2}\right)^2}}\\
&= &\frac{\sqrt{2}}{\sqrt{1+\cos^2(2\theta)}}\label{polar_cos}
\end{eqnarray}
Let us call the $\theta$ coordinate of the point $B$, $\theta=\alpha$, and the $\theta$ coordinate of the point $C$,
$\theta=\beta$. Since both points lie on the lemniscate, their $r$ coordinates satisfy $r=\sqrt{\cos(2\theta)}$. Hence the polar
coordinates of $B$ and $C$ are $(r,\theta)=(\sqrt{\cos(2\alpha)},\alpha)$ and $(r,\theta)=(\sqrt{\cos(2\beta)},\beta)$ respectively.
We are given that the $r$ coordinate of $C$ is the square of the $r$ coordinate of $B$.  Hence
\begin{equation}\label{cos_relation}
\sqrt{\cos(2\beta)}=\left(\sqrt{\cos(2\alpha)}\right)^2=\cos(2\alpha)\quad\Longrightarrow\quad\cos(2\beta)=\cos^2(2\alpha).
\end{equation}
On the lemniscate we have
\[\frac{dr}{d\theta}=-\frac{\sin(2\theta)}{\sqrt{\cos(2\theta)}},\]
so the length of the arc from $P$ to $C$ is
\begin{eqnarray*}
l &= &\int_{0}^{\beta}\sqrt{r^2+\left(\frac{dr}{d\theta}\right)^2}\,d\theta\\
&= &\int_0^\beta\sqrt{\cos(2\theta)+\frac{\sin^2(2\theta)}{\cos(2\theta)}}\,d\theta\\
&= &\int_0^\beta\frac{1}{\sqrt{\cos(2\theta)}}\,d\theta.
\end{eqnarray*}
For the area of the squircular sector we have
\begin{eqnarray*}
a &= &\frac{1}{2}\int_0^\alpha r^2\,d\theta\\
&= &\frac{1}{\sqrt{2}}\int_0^\alpha \frac{1}{\sqrt{1+\cos^2(2\theta)}}\,d\theta
\end{eqnarray*}
Hence
\[l-2a\sqrt{2}=\int_0^\beta\frac{1}{\sqrt{\cos(2\theta)}}\,d\theta-\int_0^\alpha \frac{2}{\sqrt{1+\cos^2(2\theta)}}\,d\theta.\]
Differentiating with respect to $\alpha$ and applying the fundamental theorem of calculus, the chain rule and equation (\ref{cos_relation}), we obtain
\begin{eqnarray}
\frac{d}{d\alpha}\left(l-2a\sqrt{2}\right)&=&\frac{1}{\sqrt{\cos(2\beta)}}\left(\frac{d\beta}{d\alpha}\right)-\frac{2}{\sqrt{1+\cos^2(2\alpha)}}
\\
&=&\frac{1}{\cos(2\alpha)}\left(\frac{d\beta}{d\alpha}\right)-\frac{2}{\sqrt{1+\cos^2(2\alpha)}}\label{diff_relation}
\end{eqnarray}
Implicitly differentiating with respect to $\alpha$ the relation (\ref{cos_relation})
\[\cos(2\beta)=\cos^2(2\alpha),\]
we obtain
\[-2\sin(2\beta)\left(\frac{d\beta}{d\alpha}\right) = -4\cos(2\alpha)\sin(2\alpha),\]
from which we deduce
\begin{eqnarray*}
\frac{d\beta}{d\alpha} &= &\frac{2\cos(2\alpha)\sin(2\alpha)}{\sin(2\beta)}\\
&= &\frac{2\cos(2\alpha)\sqrt{1-\cos^2(2\alpha)}}{\sqrt{1-\cos^2(2\beta)}}\\
&= &\frac{2\cos(2\alpha)\sqrt{1-\cos^2(2\alpha)}}{\sqrt{1-\cos^4(2\alpha)}}\\
&= &\frac{2\cos(2\alpha)}{\sqrt{1+\cos^2(2\alpha)}}.
\end{eqnarray*}
Substituting this back into equation (\ref{diff_relation}) we obtain
\[\frac{d}{d\alpha}\left(l-2a\sqrt{2}\right)=\frac{1}{\cos(2\alpha)}\left(\frac{2\cos(2\alpha)}{\sqrt{1+\cos^2(2\alpha)}}\right)
-\frac{2}{\sqrt{1+\cos^2(2\alpha)}}=0.\]
Hence
\[l-2a\sqrt{2}=c,\]
for some constant $c$. If we take $\alpha=0$, then $B=C=P$ and $\beta=0$, so
\[c=\left.l-2a\sqrt{2}\right|_{\alpha=0}=\int_0^0\frac{1}{\sqrt{\cos(2\theta)}}\,d\theta-\int_0^0\frac{2}{\sqrt{1+\cos^2(2\theta)}}\,d\theta=0.\]
Hence $l-2a\sqrt{2}=0$ for all $\alpha\in\left[0,\frac{\pi}{4}\right]$.
\end{proof}

\begin{corollary}\label{squircle_area}The area of the squircle is $\varpi\sqrt{2}\approx 3.70814935$, where 
$\varpi\approx2.62205755$ denotes the arc length of one lobe of the lemniscate.
\end{corollary}

\begin{proof}
When $\theta=\frac{\pi}{4}$, the shaded region in Figure 1 encompasses the entire interior of the squircle.  We then have $B=C=O$ and $l=\frac{\varpi}{2}$.
Hence
\[a=\frac{l}{2\sqrt{2}}=\frac{\varpi}{4\sqrt{2}}.\]
Thus the area of the squircle is
\[A=8a=8\left(\frac{\varpi}{4\sqrt{2}}\right)=\varpi\sqrt{2}.\]
\end{proof}

\begin{remark}
The constant $\varpi$ is commonly known as the \textit{lemniscate constant} and has connections to many interesting areas of mathematics, c.f. \cite{Wikipedia:lemniscate_constant},\cite{Wikipedia:squircle}, \cite{mparker}.
\end{remark}

\begin{remark}The relation $\overline{OC}=\overline{OB}^2$ in Theorem \ref{main_result} was discovered by the second author 
through numerical computations. The previously known result relating the perimeter of the lemniscate to the area of the squircle of
Corollary \ref{squircle_area} suggests that this should generalize to a result relating the shaded area in Figure 1 corresponding to a
given point $B$, scaled by a factor of $\sqrt{2}$, to the length of the bolded arcs for {\em some} point $C$. To discover the relation
between $B$ and $C$,  one computes the quantities $l$ and $2a\sqrt{2}$  in Theorem \ref{main_result} for various values of the
angles $\alpha$ and $\beta$ and observes that they are equal precisely when $\overline{OC}=\overline{OB}^2$.
\end{remark}

\begin{remark}\label{lemni_squircle_physics}
The proof of Theorem \ref{main_result} given above has the following physics interpretation. If point $B'$ moves counterclockwise
around the squircle so that the radial area is swept out at a constant rate of $\frac{1}{2}$, then the corresponding point $C$ moves
 around the lemniscate at a constant speed of $\sqrt{2}$.  In other words, keplerian motion on the squircle corresponds to uniform
motion along the lemniscate. This is nicely illustrated in the afore-mentioned
\href{https://www.youtube.com/watch?v=mAzIE5OkqWE&t=3s}{video}. It follows that the correspondence $B'\mapsto C$ extends
to a map from the entire squircle to the lemniscate. This map is a double branched cover of the squircle to the lemniscate -- as $B'$
goes once around the squircle, the point $C$ moves twice around the lemniscate.
\end{remark}

\begin{remark}\label{keplerian_conjecture}
Remark \ref{lemni_squircle_physics} suggests the following question. For which values of $p$  does there exist another closed curve such that keplerian motion around the $p$-norm circle $\{(x,y)\ |\ |x|^p+|y|^p=1\}$ corresponds to uniform motion around this
hypothetical curve, via some relation between the radial polar coordinates of these curves?
The answer is positive for $p=2$ where a point in keplerian motion around the circle is also in uniform motion. 
Remark \ref{lemni_squircle_physics} shows the answer is positive for $p=4$ where point $B'$ in keplerian motion around the squircle
corresponds to point $C$ in uniform motion around the lemniscate. More generally one can ask if there are any other interesting
examples of pairs of curves where keplerian motion along the first corresponds to uniform motion along the second.
Some evidence in this direction is the following generalization \cite{gemini} of equation (\ref{lemnisquircle_eqn}) 
\begin{equation}
\int_0^1 \frac{dr}{\sqrt{1-r^p}}=2^{\frac{2}{p}}\int_0^1\sqrt[p]{1-x^p}\,dx.
\end{equation}
This equation relates the area of the $p$-norm circle to the arc length of a sinusoidal spiral.
\end{remark}

\begin{remark}As we noted in the introduction, calculations which establish the relation between the perimeter of the lemniscate and
the area of the squircle date back to Legendre and Dirichlet. Going back to the definition of lemniscate by J. Bernoulli, it was known that
its perimeter is given by $4\int_0^1\frac{1}{\sqrt{1-r^4}}\,dr$ (c.f. Remark \ref{parametric_lemniscate} below). According to
 \cite[p. 524]{whittaker-watson}, Legendre \cite{legendre} computed that
\begin{equation}\label{historic_eqn}
4\int_0^1\frac{1}{\sqrt{1-r^4}}\,dr = \frac{1}{\sqrt{2\pi}}\left[\Gamma\left(\frac{1}{4}\right)\right]^2,
\end{equation}
where $\Gamma(z)=\int_0^\infty t^{z-1}e^{-t}dt$ denotes the gamma function, an extension of the factorial function. Similarly the
area enclosed by the squircle is given by $4\int_0^1\sqrt[4]{1-x^4}\,dx$. According to \cite[p. 258-259]{whittaker-watson}, an integration
method due to Dirichlet \cite{dirichlet} can be applied to evaluate this integral in terms of the gamma function. When scaled by a
factor of $\sqrt{2}$,  this gives the same result as (\ref{historic_eqn}).  In \cite{grammel} Grammel defined the $p$-norm analogs
of the trigonometric functions and explicitly computed the areas of the $p$-norm circles, including the squircle as a special case.
A similar computation was made in \cite{shelupsky}.
Apparently the relation between the perimeter of the lemniscate and the area of the squircle was first observed by Levin \cite{alevin}.
An explicit calculation relating these two quantities may be found in \cite[p. 173]{squigonometry_book}.
\end{remark}

\section{Analytic Interpretations}

The results in the previous section were presented geometrically as relations between arc lengths and areas.  In this section we
reinterpret these results analytically as relations between the analogs of the trigonometric functions associated to the lemniscate
and those associated to the squircle.  We begin by rewriting the integral for the arc length of the lemniscate

\begin{remark}\label{parametric_lemniscate}
By rewriting the polar equation of the lemniscate in Cartesian coordinates and applying the half-angle formulas, we obtain the following
parametric equations for the first quadrant portion of the lemniscate:
\begin{equation}\label{radial_parametrization}
(X,Y)=\left(r\sqrt{\frac{1+r^2}{2}},r\sqrt{\frac{1-r^2}{2}}\right),\quad 0\le r\le 1,
\end{equation}
where $r$ is the radial polar coordinate. Thus after some algebraic simplifications, we obtain the following formula for the length of
the arc of the lemniscate between $P=(1,0)$ at $r=1$ and a first quadrant point of the lemniscate with $r=R$:
\begin{equation}\label{radial_alength}
l=\int_R^1\sqrt{\left(\frac{dX}{dr}\right)^2+\left(\frac{dY}{dr}\right)^2}\,dr=\int_R^1\frac{1}{\sqrt{1-r^4}}\,dr.
\end{equation}
\end{remark}

Before we present our results, we need to give a brief overview of the lemniscate cosine. As we indicated in the introduction, we first consider a particle moving at unit speed counterclockwise around the
lemniscate in the first quadrant starting at $P=(1,0)$. Then the lemniscate cosine function $\cnl(t)$ is the distance $r$ of the particle
from the origin at time $t$. In order to make this definition rigorous, we first define the lemniscate arccosine function 
$\acl:[0,1]\longrightarrow\left[0,\frac{\varpi}{2}\right]$ by
\[t=\acl(R)=\int_R^1\frac{1}{\sqrt{1-r^4}}\,dr,\]
as in equation (\ref{radial_alength}). We then define $\cnl$ on the interval $\left[0,\frac{\varpi}{2}\right]$ to be the inverse of this
function and use symmetries of the lemniscate and periodicity to extend the domain of $\cnl$ to all real numbers. Note that by
construction the period of $\cnl$ is $2\varpi$, the perimeter of the lemniscate.

With this notation, Theorem \ref{main_result} can be reinterpreted as the following result relating the squircle and lemniscate analogs
of the trigonometric functions.

\begin{theorem} For all real numbers $t$:
\begin{equation}\label{squig_relation}
\cnl(\sqrt{2}t)=\frac{\cos_4^2(t)-\sin_4^2(t)}{\cos_4^2(t)+\sin_4^2(t)}.
\end{equation}
\end{theorem}

\begin{proof}
 Consider the points $B'$ and $C$ as moving on the squircle and lemniscate as in Remark \ref{lemni_squircle_physics} and compare Figure 1 with the following picture:
\begin{figure}[H]
\centerline{
\qquad\qquad\qquad\qquad\includegraphics[scale=3.5]{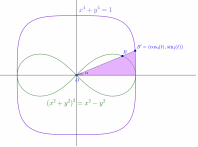}
}
\caption{Squircle and Lemniscate Trigonometry\qquad\qquad\qquad\qquad\qquad\qquad\qquad}
\end{figure}
Note that the identification $B'=(\cos_4(t),\sin_4(t))$ comes from the definition of $p$-norm trigonometric functions given in the
introduction.

If $C$ were moving at unit speed its distance from the origin at time $t$ would be $\cnl(t)$. However since it is moving at a constant
speed of $\sqrt{2}$, its distance from the origin at time $t$ is $\cnl(\sqrt{2}t)$.  Thus $\overline{OC}=\cnl(\sqrt{2}t)$.

Also we see that the point $B$ in Figure 1 is the radial projection of the point $B'=(\cos_4(t),\sin_4(t))$ in Figure 2 onto the lemniscate.
Hence
\[B=\left(\frac{\cos_4(t)\sqrt{\cos_4^2(t)-\sin_4^2(t)}}{\cos_4^2(t)+\sin_4^2(t)},
\frac{\sin_4(t)\sqrt{\cos_4^2(t)-\sin_4^2(t)}}{\cos_4^2(t)+\sin_4^2(t)}\right).\]
Substituting into the relation  $\overline{OC}=\overline{OB}^2$
and simplifying, we obtain:
\[
\cnl(\sqrt{2}t)=\frac{\cos_4^2(t)-\sin_4^2(t)}{\cos_4^2(t)+\sin_4^2(t)}.
\]
A priori we have only shown this holds in the range $0\le t\le\frac{\varpi}{2\sqrt{2}}$. However the physics interpretation of the motion
of $B'$ and $C$ given in Remark \ref{lemni_squircle_physics} shows that this extends to all real values $t$. This equation analytically
encodes the geometric interpretation given in Theorem \ref{main_result}. In particular we can derive the area formula for the
squircle of Corollary \ref{squircle_area} by comparing the periods of the two sides of this equation.
\end{proof}
\bigskip

We can obtain a different relation between the lemniscate and squircle trigonometric functions by deriving an alternative integration
formula for the area of a squircle. If we use the polar equation (\ref{polar_tan}), instead of (\ref{polar_cos}), to compute the area $a$ of the squircle sector depicted in Figure 2, we obtain
\begin{equation}\label{hyperbolic_area1}
t=2a=\int_0^\alpha r^2\,d\theta=\int_0^\alpha \frac{\sec^2(\theta)}{\sqrt{1+\tan^4(\theta)}}\,d\theta.
\end{equation}
If we substitute $v=\tan(\theta)$ in this integral, we get
\begin{equation}\label{hyperbolic_area2}
t=\int_0^{\tan(\alpha)}\frac{dv}{\sqrt{1+v^4}}.
\end{equation}
This relates the squircle area to the integral $\int_0^V\frac{dv}{\sqrt{1+v^4}}$ instead of $\int_R^1\frac{1}{\sqrt{1-r^4}}\,dr$.
Now the function $V\mapsto\int_0^V\frac{dv}{\sqrt{1+v^4}}$ is increasing on the real line and thus has an inverse
which is known as the \textit{hyperbolic lemniscate sine} $\slh(t)$ (c.f. \cite{Wikipedia:lemniscate_elliptic_functions}). Thus the equation
(\ref{hyperbolic_area2}) can be written as $t=\slh^{-1}(\tan(\alpha))$.  Hence we obtain
\begin{equation}\label{hyperbolic_eqn}
\tan_4(t)=\tan(\alpha)=\slh(t),
\end{equation}
where $\tan_4(t)=\frac{\sin_4(t)}{\cos_4(t)}$. However the hyperbolic lemniscate sine function can be expressed in terms of
the regular lemniscate trigonometric functions:
\[\slh(t)=\frac{\left(1+\cnl^2\left(\frac{t}{\sqrt{2}}\right)\right)\snl\left(\frac{t}{\sqrt{2}}\right)}{\sqrt{2}\cnl\left(\frac{t}{\sqrt{2}}\right)}\]
Combining this with equation (\ref{hyperbolic_eqn}) we obtain
\begin{equation}
\tan_4(t)=\frac{\left(1+\cnl^2\left(\frac{t}{\sqrt{2}}\right)\right)\snl\left(\frac{t}{\sqrt{2}}\right)}{\sqrt{2}\cnl\left(\frac{t}{\sqrt{2}}\right)},
\end{equation}
where $\snl$ is the lemniscate sine function, related to the lemniscate cosine function by the identity
\[\snl(u)=\cnl\left(\frac{\varpi}{2}-u\right),\]
by analogy to the trigonometric identity
\[\sin(u)=\cos\left(\frac{\pi}{2}-u\right).\]
For further reference we observe that this definition amounts to the statement that the lemniscate sine function is obtained by
inverting the integral $\int_0^R\frac{1}{\sqrt{1-r^4}}\,dr$, i.e. considering a particle moving at constant speed around the lemniscate
starting at the origin instead of $(1,0)$ as in the case of the lemniscate cosine function. Moreover the lemniscate analog of the Pythagorean identity $\cos^2(u)+\sin^2(u)=1$ is the identity
\begin{equation}\label{lemni_pythagorean}
\cnl^2(u)+\snl^2(u)+\cnl(u)\snl(u)=1.
\end{equation}

\section{An Alternate Geometric Relation and Proof}

In his lectures on complex function theory, Siegel \cite[p. 4, equation (9)]{siegel} derives the following relation between integrals:
\begin{equation}\label{Siegel_relation}\
\int_0^R\frac{dr}{\sqrt{1-r^4}}=\sqrt{2}\int_0^T\frac{dt}{\sqrt{1+t^4}},\quad R=\frac{\sqrt{2}\,T}{\sqrt{1+T^4}}.
\end{equation}
Siegel's calculation was a step in his explanation of Fagnano's theorem on doubling the arc of a lemniscate, which was the beginning of the
theory of elliptic integrals and elliptic functions, and was not connected in any way in his mind to the analysis of the squircle.
However in view of equation (\ref{hyperbolic_area2}), we see that this equation relates the
length of an arc of the lemniscate with the area of a sector of the squircle. This is a different relation than the one given in Theorem \ref{main_result}.
However these results are strongly related and in fact Siegel's calculation provides an alternative proof of Theorem \ref{main_result}.

We begin with a geometric interpretation of equation (\ref{Siegel_relation}) shown in the figure below.
\begin{figure}[H]
\centerline{
\includegraphics[scale=4]{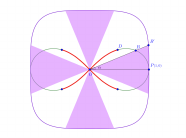}
}
\caption{Another Relation}
\end{figure}

The relation depicted here
is $L=A\sqrt{2}$, where $A$ is the shaded area in the picture and $L$ is the arc length of the bolded portion of the lemniscate.
This is exactly the same relation as depicted in Figure 1, except different parts of the lemniscate are bolded.
 The points $B$ and $D$ in the picture are related by the equation ${\overline{OD}^2=\frac{1-\overline{OB}^4}{1+\overline{OB}^4}}$, 
as opposed to the relation $\overline{OC}=\overline{OB}^2$ in Figure 1. Note that the area depicted in Figure 3 is exactly the same as in Figure 1.
Hence, given Theorem \ref{main_result}, the arc length $L$ of the bolded portion of the lemniscate in Figure 3 must equal the arc length 
of the different bolded portion in Figure 1. [We give an independent proof of this below  in Proposition \ref{arc_equality}.]
As in Figure 1, due to the symmetries of the squircle and the lemniscate, it suffices to check the relation for the portions of the areas and lengths 
of Figure 3 lying in the first octant: ${\{(x,y)\ |\ 0\le y\le x\}}$. If we denote the corresponding area by $a$ and the corresponding arc length by $l$, 
we have as before $A=8a$ and $L=4l$ and the given relation reduces to $l=2a\sqrt{2}$.

Hence we are reduced to proving the alternative version of Theorem \ref{main_result}.

\begin{theorem}\label{alternate_result} Let $B$ be a point in the first quadrant of the lemniscate $\left(x^2+y^2\right)^2=x^2-y^2$ and let $B'$ 
be its radial projection onto the squircle $x^4+y^4=1$. Let $D$ be the point on the first quadrant of the lemniscate such that 
${\overline{OD}^2=\frac{1-\overline{OB}^4}{1+\overline{OB}^4}}$. Then $l=2a\sqrt{2}$
where $l$ is the arc length of the lemniscate from $O$ to $D$ and $a$ denotes the area of the squircular sector $OPB'$.
\end{theorem}

\begin{proof}
Siegel \cite[p. 3-4]{siegel} shows that substitution
\begin{equation}\label{Siegel_substitution}
r=\frac{\sqrt{2}\,v}{\sqrt{1+v^4}}
\end{equation}
proves the equality of integrals
\begin{equation}\label{Siegel_relation2}
\int_0^{R'}\frac{dr}{\sqrt{1-r^4}}=\sqrt{2}\int_0^T\frac{dv}{\sqrt{1+v^4}},\quad R'=\frac{\sqrt{2}\,T}{\sqrt{1+T^4}},\ 0\le T\le 1.
\end{equation}
[Note that we have renamed the variable of integration on the right side of equation (\ref{Siegel_relation2}) to avoid notational clashes in what follows below.]
In view of equation (\ref{hyperbolic_area2}), if we take $T=\tan(\alpha)$, the right hand side of equation (\ref{Siegel_relation2}) is precisely $2a\sqrt{2}$.
Hence it remains to show that the left hand side of equation (\ref{Siegel_relation2}) is the length of the lemniscate arc from $O$ to $D$, which
amounts to showing that $R'=\overline{OD}$.

If we denote $R=\overline{OB}$, then by equation (\ref{radial_parametrization}), the point $B$ has Cartesian coordinates
\[(x,y)=\left(R\sqrt{\frac{1+R^2}{2}},R\sqrt{\frac{1-R^2}{2}}\right).\]
Hence
\[T=\tan(\alpha)=\frac{y}{x}=\sqrt{\frac{1-R^2}{1+R^2}}.\]
It follows that
\begin{eqnarray*}
(R')^2 &= &\frac{2T^2}{1+T^4}\\
&= &\frac{2\left(\frac{1-R^2}{1+R^2}\right)}{1+\left(\frac{1-R^2}{1+R^2}\right)^2}\\
&= &\frac{1-R^4}{1+R^4}\\
&= &\frac{1-\overline{OB}^4}{1+\overline{OB}^4}\\
&= &\overline{OD}^2
\end{eqnarray*}
Thus $R'=\overline{OD}$ and we are done.
\end{proof}

We conclude by showing that Theorem \ref{alternate_result} implies Theorem \ref{main_result}. First observe that the proof of Theorem \ref{alternate_result} above
is mostly independent of the previous proof of Theorem \ref{main_result}. The only overlap is the derivation of the polar equation (\ref{polar_tan}) of the squircle,
which was then used in equations (\ref{hyperbolic_area1}) and (\ref{hyperbolic_area2}) to compute the area $a$ of the squircle sector. Thus we are not
engaging in circular reasoning when we derive Theorem \ref{main_result} from Theorem \ref{alternate_result} and the following Proposition.

\begin{proposition}\label{arc_equality}
The length of the lemniscate arc from $O$ to $D$ in Figure 3  is equal to the length of the arc from $C$ to $P$ in Figure 1.
\end{proposition}

\begin{proof}
Let $l_1$ denote the length of the arc from $O$ to $D$ and $l_2$ denote the length of the arc from $C$ to $P$. Then by the definition of the lemniscate sine and
cosine functions, we have $\snl(l_1)=\overline{OD}$ and $\cnl(l_2)=\overline{OC}=\overline{OB}^2$. By the lemniscate Pythagorean identity (\ref{lemni_pythagorean}),
\[\snl^2(l_2)=\frac{1-\cnl^2(l_2)}{1+\cnl^2(l_2)}=\frac{1-\overline{OC}^2}{1+\overline{OC}^2}=\frac{1-\overline{OB}^4}{1+\overline{OB}^4}=\overline{OD}^2=\snl^2(l_1).\]
Since $\snl$ is non-negative and injective on the interval $\left[0,\frac{\varpi}{2}\right]$, it follows that ${l_1=l_2}$.
\end{proof}

\bigskip

\begin{remark}
We can also prove Theorem \ref{main_result} directly by adapting the proof used for Theorem \ref{alternate_result}. Observe that Theorem \ref{main_result} may be
written as
\begin{equation}
\int_{R^2}^1\frac{dr}{\sqrt{1-r^4}}=\sqrt{2}\int_0^T\frac{dv}{\sqrt{1+v^4}},\quad T=\tan(\alpha)=\sqrt{\frac{1-R^2}{1+R^2}},
\end{equation}
and these two integrals are related by the substitution $v=\sqrt{\frac{1-r}{1+r}}$. We leave the details of this proof as an exercise for the reader.
\end{remark}


\bigskip

\begin{thebibliography}{99}
\bibitem{apostol}Apostol, Tom M. (1991). \textit{Calculus}, Volume 1. ISBN 978-0471000051.
\bibitem{dirichlet}Dirichlet, P. G. L., ``Ueber eine neue Methode zur Bestimmung vielfacher Integrale'', 1839. In: Kronecker L, ed. \textit{G. Lejeune Dirichlet’s Werke,} Cambridge Library Collection - Mathematics. Cambridge University Press; 2012:381-390.
\href{https://ia801605.us.archive.org/23/items/glejeunedirichl01dirigoog/glejeunedirichl01dirigoog.pdf}{\faExternalLink*}
\bibitem{gambini_et_al}Gambini, A., Nicoletti, G. \& Ritelli, D.(2021). Keplerian trigonometry. Monatsh Math 195, 55–72.
\href{https://link.springer.com/content/pdf/10.1007/s00605-021-01512-0.pdf}{\faExternalLink*}
\bibitem{gemini}Google (2025, Dec. 7) Gemini Pro, \href{https://gemini.google.com}{\faExternalLink*}
\bibitem{grammel}Grammel, Richard (1948). ``Eine Verallgemeinerung der Kreis- und Hyperbelfunktionen''. Arch. Math 1, 47–51.
\bibitem{legendre}Legendre, A-M, \textit{Exercices du calcul intégral,} Paris, Huzard-Coursier, 1811. \href{https://ia601305.us.archive.org/14/items/exercicesdecalc00legegoog/exercicesdecalc00legegoog.pdf}{\faExternalLink*}
\bibitem{alevin}Levin, Aaron (2006). ``A Geometric Interpretation of an Infinite Product for the Lemniscate Constant''. The American Mathematical Monthly. 113 (6): 510–520. \href{https://web.archive.org/web/20041220213524id_/http://math.berkeley.edu:80/~adlevin/Lemniscate.pdf}{\faExternalLink*}
\bibitem{mparker}Parker, Matt (2021). ``What is the area of a Squircle?''. Stand-up Maths. YouTube. \href{https://www.youtube.com/watch?v=gjtTcyWL0NA}{\faExternalLink*}
\bibitem{squigonometry_book}Poodiack, Robert D.; Wood, William E. (2022). \textit{Squigonometry: The Study of Imperfect Circles}. Springer. ISBN 978-3031137822.
\bibitem{poodiack}Poodiack, Robert D. (2016). ``Squigonometry, hyperellipses, and supereggs''. Mathematics Magazine. 89(2):92–102.
\href{https://www.researchgate.net/publication/303865545_Squigonometry_Hyperellipses_and_Supereggs}{\faExternalLink*}
\bibitem{mvr}Ramalingam, Muthu Veerappan (2023). ``Bernoulli Lemniscate and the Squircle ${\displaystyle ||}$ A remarkable Geometric fun fact!!?''. 
Act of Learning. YouTube. \href{https://www.youtube.com/watch?v=mAzIE5OkqWE&t=3s}{\faExternalLink*}
\bibitem{siegel}Siegel, Carl Ludwig (1969). \textit{Topics in Complex Function Theory -- Elliptic Functions and Uniformization Theory}, Volume 1, Wiley. ISBN 9780471608448.
\bibitem{spivak}Spivak, Michael (2008). \textit{Calculus}, 4th edition. ISBN 978-0914098911.
\bibitem{shelupsky}Shelupsky, David (1959). ``A Generalization of the Trigonometric Functions''. The American Mathematical Monthly, 66(10), 879–884.
\bibitem{whittaker-watson}Whittaker E.T.; WatsonG. N.(1978), \textit{A Course of Modern Analysis,} Cambridge University. Press, Cambridge. ISBN 0521-09189-6.
\bibitem{Wikipedia:gamma_function}Wikipedia Contributors, Gamma function, Wikipedia, The Free 
Encyclopedia, retrieved November, 2024. \href{https://en.wikipedia.org/wiki/Gamma_function}{\faExternalLink*}
\bibitem{Wikipedia:lemniscate_constant}Wikipedia Contributors, Lemniscate constant, Wikipedia, The Free
Encyclopedia, retrieved November, 2024. \href{https://en.wikipedia.org/wiki/Lemniscate_constant}{\faExternalLink*}
\bibitem{Wikipedia:lemniscate_elliptic_functions}Wikipedia Contributors, Lemniscate elliptic functions, Wikipedia, The Free
Encyclopedia, retrieved November, 2024. \href{https://en.wikipedia.org/wiki/Lemniscate_elliptic_functions}{\faExternalLink*}
\bibitem{Wikipedia:lemniscate_of_bernoulli}Wikipedia Contributors, Lemniscate of Bernoulli, Wikipedia, The Free
Encyclopedia, retrieved November, 2024. \href{https://en.wikipedia.org/wiki/Lemniscate_of_Bernoulli}{\faExternalLink*}
\bibitem{Wikipedia:squigonometry}Wikipedia Contributors, Squigonometry, Wikipedia, The Free Encyclopedia, retrieved November, 2024.
\href{https://en.wikipedia.org/wiki/Squigonometry}{\faExternalLink*}
\bibitem{Wikipedia:squircle}Wikipedia Contributors, Squircle, Wikipedia, The Free Encyclopedia, retrieved November, 2024.
\href{https://en.wikipedia.org/wiki/Squircle}{\faExternalLink*}
\bibitem{wood}Wood, William E., 2011. ``Squigonometry''. Mathematics Magazine, 84(4):257–265.
\end{thebibliography}
\end{document}